\documentclass[11pt]{amsart}
\usepackage{amsfonts}
\usepackage{amssymb}
\usepackage{graphicx}%

\usepackage{color,xcolor} 
\usepackage[utf8]{inputenc} 

\usepackage{amscd}
\usepackage{mathrsfs}
\newtheorem{theorem}{Theorem}[section]

\newtheorem{corollary}[theorem]{Corollary}
\newtheorem{definition}[theorem]{Definition}

\newtheorem{remark}[theorem]{Remark}

\def\wh{\widehat}
\def\wt{\widetilde}

\def\DO{\mathcal D}
\def\RE{\mathbb R}
\def\CO{{\mathbb C}}
\def\N{\mathbb N}

\def\C{\mathcal C}
\def\L{\mathcal L}
\def\K{\mathcal K}
\def\A{\mathcal A}
\def\W{\mathcal W}
\def\F{\mathcal F}

\begin{document}

\title[ODEs with point interactions: An
inverse problem]{Ordinary differential equations with point
interactions: An inverse problem}

\author[Dias, Jorge and Prata]{Nuno Costa Dias, Cristina Jorge and Jo\~{a}o Nuno Prata}

\begin{abstract}
Given a linear ordinary differential equation (ODE) on $\RE$ and a set of interface conditions at a finite set
of points $I \subset \RE$, we consider the problem of determining another differential equation whose {\it global}
solutions satisfy the original ODE on $\RE \backslash I $, and the interface conditions at $I $. Using an
extension of the product of distributions with non-intersecting singular supports presented in [L. H\"ormander,
The Analysis of Linear Partial Diffe\-rential Operators I, Springer-Verlag, 1983], we determine an
{\it intrinsic} solution of this problem, i.e. a new ODE, satisfying the required conditions, and strictly
defined within the space of Schwartz distributions. Using the same formalism, we determine a singular perturbation formulation for the $n$-th order derivative operator with interface conditions.

\end{abstract}

\maketitle

\keywords{{\bf Keywords:} Linear ordinary differential equation with distributional coefficients; multiplicative products of distributions; point interactions }

\section{Introduction}

Let ODE$_1$ be a linear ordinary differential equation of the form
\begin{equation} \label{ODE1}
\sum_{i=0}^n a_i D^{i}\psi=f
\end{equation}
where $a_i,f\in\C^{\infty}(\RE)$, and ${a_n}(x) \ne 0$ for all $x \in \mathbb{R}$. Moreover, $D^{i}
\psi=\psi^{(i)}$ is the $i$th order distributional derivative of $\psi$. Let also ${ I}\subset \RE$
be a finite set.

In this paper we consider the following problem:\\

\noindent\textbf{Problem 1.} {\it Let ODE$_1$ and ${ I}$ be defined as above. We want to determine a new
ordinary differential equation (let us denote it by ODE$_2$) whose global solutions $\psi$ satisfy:

\begin{enumerate}

\item[(C1)] The ODE$_1$ on $\RE \backslash { I}$.

\item[(C2)] At each point $x_0 \in { I}$, the interface conditions:
\begin{equation}\label{importantef}
A(x_0)\overline{\psi({x_0^ -})} =B(x_0)\overline{\psi({x_0^ + })}
\end{equation}
where
\begin{equation}\label{77}
A(x_0)=\left[A_{ij}(x_0)\right],\ \ B(x_0)=\left[B_{ij}(x_0)\right]
\end{equation}
are $m \times n$ (in general complex valued) matrices with $m \le
n$, and
\begin{equation} \label{overbarpsi}
\overline{\psi({x_0^ \pm })}=\lim_{\epsilon \to 0^+}(\psi(x_0 \pm \epsilon),...,D^{n-1}\psi(x_0\pm \epsilon))^T
\, .
\end{equation}

\end{enumerate}

}

At each point $x_0 \in I$, the conditions (\ref{importantef}) can be {\it separating} or {\it interacting}. In
the separating case they reduce to a set of conditions for $\overline{\psi(x_0^{\mathbin{+}})}$ and
another set of conditions for $\overline{\psi(x_0^{\mathbin{-}})}$. In this case the values of
$\overline{\psi(x_0^{\mathbin{+}})}$ and $\overline{\psi(x_0^{\mathbin{-}})}$ are
independent of each other. In the interacting case the conditions relate the values of
$\overline{\psi(x_0^{{+}} )}$ with those of $\overline{\psi(x_0^{{-} })}$. If they do
not completely fix the values of $\overline{\psi(x_0^{{-}})}$ in terms of those of
$\overline{\psi(x_0^{{+}} )}$ or vice-versa, we say that the conditions are only {\it partially
interacting}. We remark that in the case of separating conditions, the Cauchy problem for the ODE$_2$ 
may have no solutions for some initial conditions, and the solution may not be unique for other initial conditions (this can also be the case for partially interacting conditions - see Corollary \ref{Corollary4.6} for an example).

The aim of this paper is to obtain the explicit form of the equation ODE$_2$. 
A possible approach is to allow this equation to display
distributional (maybe singular) coefficients. 
Differential equations with singular (or just discontinuous) 
coefficients have been studied using several different approaches. 
Most significative are the intrinsic formulations, defined in terms of a suitable 
product of Schwartz distributions \cite{{Bag95,Cad08,DP09,DPJ16,HO07,Obe92,Sarrico}} and the formulations
using generalized functions
\cite{{Col84,Col85,GKOS01,HO09}}.  The latter approach uses objects that are more general than the Schwartz distributions, and yields the most
general, but also the most complex formulation of these 
differential problems. The intrinsic products, on the other hand,
provide formulations that are simpler, but often limited to very
particular cases. The model product, for instance, which is the
most general product in the hierarchy given by
Oberguggenberger in (section 7, \cite{Obe92}), has been used
to formulate ODEs with discontinuous coefficients and displaying
non-smooth solutions, but is not compatible with the case of
singular coefficients and discontinuous solutions (cf.
\cite{HH08,HO07}).

A closely related problem from the theory of differential operators, is the formulation of singular perturbations of one-dimensional Schr\"odinger operators \cite{Albeverio1,Albeverio2,Brasche,DPP11,Exner,Golovaty,Kurasov} (particularly of the free Schr\"odinger operator $\wh H_0=-D^2$), and also of the $n$-th order derivative operator $\wh L_0=(iD)^n$ \cite{KurasovBoman}. This topic has been thoroughly study in the literature in connection with the problem of relating the singular perturbations with self-adjoint (s.a.) interactions, i.e. with the s.a. realizations of $\wh H_0$ (and $\wh L_0$) in domains of functions satisfying some interface conditions at some finite set of points.    

Important results, in this context, were obtained by Kurasov and Boman. They determined an explicit formulation of the s.a. singular perturbations of $\wh H_0$ \cite{Kurasov} and $\wh L_0$ \cite{KurasovBoman} (see also \cite{Albeverio2,Brasche}) using a theory of distributions acting on discontinuous test functions. Their formulation, however, is not intrinsic; Kurasov distributions are more general objects than Schwartz distributions  and, just like in the Colombeau case, this approach requires the new operators to be formulated in terms of the (more complex) structure of the new space of distributions. This includes a new derivative operator (that does not satisfy the Leibniz rule, the derivative of a constant is not zero, etc.), new distributions and a new dual product.

 In this paper we will work strictly within the space of Schwartz distributions. Using the intrinsic approach developed in \cite{DP09,DPJ16} we will determine a formulation of the ODE$_2$ that solves Problem 1 for all cases.
As a by-product we will also consider the operator $\wh L_0$, and determine an explicit formulation of its singular perturbations that correspond to an interface condition of the form (\ref{importantef}), with $x_0\in I=\{0\}$. This set of singular perturbations includes all s.a. extensions of the symmetric restriction of $\wh L_0$ to $\DO(\RE \backslash \{0\})$.

Our formulation is based on the intrinsic multiplicative
product of Schwartz distributions, denoted $*$, that was defined
and studied in \cite{DP09,DPJ16}. This product is an extension of the
product of distributions with disjoint singular
supports presented by H\"ormander in \cite{Hor83}. It is associative, distributive, non-commutative,
reproduces the standard product of functions for regular
distributions, and satisfies the Leibniz rule with respect to the
usual Schwartz distributional derivative.

In view of the Schwartz impossibility result \cite{Sch54}, the
product $*$ cannot be defined in the entire space $\DO'(\RE)$. In
fact, it is only defined in the subspace $\A= \cup_{i=0}^\infty
D^i [\C^{\infty}_p(\RE)]\subset\DO'(\RE)$, where
$\C^{\infty}_p(\RE)$ is the set of piecewise smooth functions. In
$\A$ the choice of the product $*$ is optimal: we have recently
proved \cite{DJP16-2} that $*$ is essentially the unique
multiplicative product defined in $\A$ that satisfies all the
properties stated in the Schwartz impossibility result. Moreover,
the space $\A$ is large enough to allow for the formulation of a
significative class of differential problems with discontinuous
solutions. In particular, it is possible to obtain a consistent
formulation of linear ODEs with coefficients and solutions in
$\A$. This, as we will see, is the case of the ODE$_2$ that will
be derived here.

The plan of the paper is the following: For the convenience of the reader, in the next section we briefly
review the definition and main
properties of the distributional product that was presented in
\cite{DP09}. We also review the definition and properties of the
operators of (left and right) ``multiplication by the $n$th-order
derivative of a Dirac delta'' that were defined in \cite{DPJ16}.
In Section \ref{new} we define a new derivative operator and study
its main properties. This operator is then used in Section
\ref{Linear} to determine a possible solution of Problem 1. We
also determine a second formulation of the ODE$_2$, which is
written in terms of the standard distributional derivative. The two formulations are equivalent.
In section \ref{Simple example}, we illustrate the formulation of the ODE$_2$ using a simple example.
In section \ref{Derivative} we use the previous results to determine an explicit formulation for the singular perturbations of the $n$-th order derivative operator. Finally, in section \ref{Conclusions} we provide a more detailed discussion of the relation between Kurasov's approach and ours, and also discuss some future work.   

In this paper we have considered the case of a linear ODE$_1$ with
smooth coefficients, and linear interface conditions. It is clear however 
from the presentation, that the main results can be extended, at least, to 
some cases of non-linear interface
conditions, non-smooth coefficients, and some
classes of non-linear ODEs.\\

\noindent\textbf{Notation}. The letters  $\Omega$ and $\Theta$
denote open subsets of $\RE$, and $I$ and $J$ are discrete sets of
real numbers. Usually, the letters $F$ and $G$ denote
distributions, $f$ and $g$ are (piecewise) smooth functions, and
$\psi$, $\phi$ may be distributions or regular functions. Letters
with a hat denote operators. The functional spaces are denoted by
calligraphic capital letters, e.g. $\A(\Omega)$,
$\C^\infty(\Omega)$, $\DO(\Omega)$,.... When $\Omega = \RE$ we
usually write only $\A$, $\C^\infty$, $\DO$,...

The characteristic function of $\Omega \subseteq \RE$ is written $\chi_\Omega$, the Heaviside step function is
$H= \chi_{\RE_{\mathbin{+}}}$, and $H_-=1-H$. As usual, $\delta(x-x_0)$ is the Dirac measure with
support at $x_0$.

\section{Preliminaries} \label {pri}

In this section we review some results about Schwartz
distributions, and present the main properties of the
multiplicative product of distributions $*$ that was proposed in
\cite{DP09}. We also review the definition of the operators of
(left and right) ``multiplication by a Dirac delta'' that were
proposed in \cite{DPJ16}. The reader should refer to
{\cite{DP09,DPJ16}} for details and proofs of the main
results.

\subsection{Algebras of Schwartz distributions}

We start by presenting some basic definitions. Let $\DO$ be the
space of smooth functions $t: \RE \to \CO$ of compact support. Its
dual $\DO'$ is the space of Schwartz distributions. The singular
support of a distribution $F\in\DO'$, denoted by $\textrm {sing
supp } F$, is  the complement of the largest open set $\Omega
\subseteq\RE$ for which there is $f\in \C^\infty(\Omega)$  such
that $F|_{\Omega}=f$ (where $F|_{\Omega}$ denotes the restriction
of $F$ to $\DO(\Omega)$). Another useful concept is the order of a
distribution [pag.43, \cite{Kan98}]: we say that $F \in \DO'$ is
of order $n$ (and write $n=$ ord $F$) iff $F$ is the $n$th order
distributional derivative (but not a lower order distributional
derivative) of a regular distribution. A distribution of order
zero is a regular distribution.

Let $\C^\infty_p$ be the space of piecewise smooth functions on $\RE$: ${ f} \in \C^\infty_p$ iff there is a
finite set ${ I}\subset \RE$ such that ${ f} \in \C^{\infty}(\RE \backslash { I})$ and the lateral limits
$\lim_{x \to x_0^{\pm}} { f} ^{(j)} { (x)}$ exist and are finite for all $x_0 \in { I}$ and all derivatives of
${f} $.

A distributional extension of the set $\C^\infty_p$ is given by:

\begin{definition} \label{c}
Let $\A$ be the space of all piecewise smooth functions $\C^\infty_p$ (regarded as regular distributions)
together with all their distributional derivatives to all orders.

Let $I \subset \RE$ be a finite set. We also define the following subspaces of $\A$:

\begin{equation} \label{subspacesA}
\A_I=\left\{ F \in \A : \, \mbox{sing supp} \, F \subseteq I \right\} \, .
\end{equation}

\end{definition}

Notice that $\A_I \subseteq \A_J \subseteq \A$ for all $I \subseteq J$. 
All the elements of $\A$ are distributions with finite singular supports. They can be written in the form:

\begin{theorem} \label{f6}
$F \in \A$ iff there is a finite set ${ I}=\{x_1<x_2<...<x_m\}$ associated with a set of open intervals
$\Omega_i=(x_i,x_{i+1})$, $i=0,..,m$ (where $x_0=-\infty$ and $x_{m+1}=+\infty$)  such that:
\begin{equation}\label{FormF}
F= f+\Delta
\end{equation}
where ($\chi_{\Omega_i}$ is the characteristic function of $\Omega_i$):
\begin{equation} \label{FormF2}
 f=\sum_{i=0}^m f_i \chi_{\Omega_i}\ \ \text{  and } \ \ \Delta=\sum_{i=1}^m \sum_{j=0}^n
{F}_{ij}\delta^{(j)}(x-x_i)
\end{equation}
for some ${ F}_{ij} \in \CO$ and $f_i \in \C^{\infty}(\RE)$. We have, of course, sing supp $F
\subseteq { I}$.
\end{theorem}

Notice, in particular, that every $F \in  \A_{\left\{0\right\}}$ can be
written in the form:
\begin{equation}\label{0}
F=H_-f_-+H f_+ +\Delta,
\end{equation}
where $H$ is the Heaviside step function, $H_- = 1-H$, $f_\pm \in \C^{\infty}(\RE)$ and $\text{supp
}\Delta\subseteq\left\{0\right\}$.

More generally, we have from Theorem \ref{f6}:

\begin{theorem}\label{18}
If ${ F} \in\A$ then there exist $\epsilon>0$ and ${f_-} ,{ f_+}
\in \C^{\infty}(\RE)$, such that:
\begin{equation} \label{Proj}
{ F} |_{\left(-\epsilon,0\right)} = {f} _- \chi_{\left(-\epsilon,0\right)} \quad \mbox{and } \quad
{F} |_{\left(0,\epsilon\right)}={ f} _+ \chi_{\left(0,\epsilon\right)} \, .
\end{equation}
\end{theorem}

\begin{proof}
If ${F} \in\A$ then by Theorem \ref{f6} there is a finite set of real numbers $I$ such that
${ F} ={f} +\Delta$ where $\Delta$ is a distribution with support on ${I}$
and ${f} \in \C^{\infty}_p\cap\C^{\infty}(\RE \backslash { I})$. Since ${
I}$ is a finite set, there is always a real number $\epsilon>0$ such that ${
I}\cap\left(-\epsilon,0\right)=\emptyset$ and ${I}\cap\left(0,\epsilon\right)=\emptyset$. Hence,
$(-\epsilon , 0) \subseteq \Omega_i$ and $(0, \epsilon) \subseteq \Omega_j$ for some $i,j \in \{0,..,m\}$ (cf.
Theorem \ref{f6}) and thus eq.(\ref{Proj}) holds for ${f} _- = f_i$ and ${ f} _+ = f_j$,
which concludes the proof.
\end{proof}

We now recall some basic definitions about products of distributions. The dual product of $F \in \DO'$ by $g \in
\C^\infty$ is defined by:
\begin{equation} \label{dual}
\langle F g, h \rangle= \langle F, gh \rangle \quad , \quad \forall \,h \in \DO \, .
\end{equation}
This product can be generalized to the
case of two distributions with finite and disjoint singular
supports [pag.55, \cite{Hor83}]. In $\A$ this generalization 
can be defined as follows:

\begin{definition} \label {Hormander}
Let $F,G \in \A$ be two distributions such that sing supp $F$ and sing supp $G$ are finite disjoint sets. Then
there exists a finite open cover of $\RE$ (denote it by $\{\Omega_i \subset \RE,\, i=1,..,d \}$) such that, on
each open set $\Omega_i$, either $F$ or $G$ is a $\C^\infty(\Omega_i)$-function. Hence, on each $\Omega_i$, the
two distributions can be multiplied using the dual product. The product $\cdot$ of $F$ and $G$ is then
defined as the unique distribution $F \cdot G \in \A$ that satisfies:
\begin{equation}\label{00}
F \cdot G|_{\Omega_i}= F|_{\Omega_i} G|_{\Omega_i} \quad , \quad  i=1,..,d.
\end{equation}
\end{definition}

The new product $*$ extends the previous product to the case of an arbitrary pair of distributions in
$\A$:

\begin{definition} \label{Productdef}
Let $F,G\in\A$. The multiplicative product $*$  is defined  by
\begin{eqnarray}\label{55}
F*G=\mathop {\lim }\limits_{\varepsilon  \downarrow 0}F(x) \cdot G(x+\epsilon)
\end{eqnarray}
 where the product $F(x)\cdot G(x+\epsilon)$ is given in the previous definition and the limit is taken in the distributional sense.
\end{definition}

Notice that for sufficiently small $\epsilon >0$, $F(x)$ and $G(x+\epsilon)$ have disjoint singular supports,
hence the $\cdot $ product in (\ref{55}) is well-defined in the sense of Definition (\ref{Hormander}).

The next theorem gives an explicit formula for $F*G$. Let $F,G \in \A$, let sing supp $F \, \cup$ sing supp $G
=\{x_1<..<x_m\}$, and consider the  associated set of open intervals $\Omega_i=(x_i,x_{i+1})$, $i=0,..,m$ (with
$x_0=-\infty$ and $x_{m+1}=+\infty$). Then, in view of Theorem \ref{f6}, $F$ and $G$ can be written in the form:
\begin{eqnarray}\label{187}
F &=& \sum_{i=0}^m f_i \chi_{\Omega_i}+\sum_{i=1}^m \sum_{j=0}^n {F} _{ij}\delta^{(j)}(x-x_i)   \nonumber \\
G &=& \sum_{i=0}^m g_i \chi_{\Omega_i}+\sum_{i=1}^m \sum_{j=0}^n { G} _{ij}\delta^{(j)}(x-x_i)
\end{eqnarray}
where $f_i,g_i \in \C^\infty$ and ${ F} _{ij}=0$ if $x_i \notin \text{sing supp F}$ or if $j \ge$ ord
$F$, and likewise for $G$.  Then
\begin{theorem}\label{2.5}
Let $F,G \in \A$ be written in the form (\ref{187}). Then $F*G$ is given explicitly by
\begin{equation} \label{prodf}
F * G = \sum_{i=0}^m f_i g_i \chi_{\Omega_i}+\sum_{i=1}^m \sum_{j=0}^n \left[ { F} _{ij} g_i  +
{ G} _{ij} f_{i-1} \right]  \delta^{(j)}(x-x_i).
\end{equation}
and so $F*G \in \A$.
\end{theorem}

Some important examples are:

$$
\delta^{(k)} (x) * H(x)= \delta^{(k)} (x)  \quad \mbox {and}  \quad H(x) * \delta^{(k)} (x)= 0 \, , \quad
\forall k \in \N_0
$$
$$
\delta^{(k)}(x) * \delta^{\mathbin{(}l)}(x) =0 \, , \quad \forall k,l \in \N_0 \, .
$$

Notice that the $*$ product of two distributions does not always display the expected symmetry. We have, for instance:
$$
\delta (x) * (-\frac{1}{2} + H(x)) = \frac{1}{2} \delta (x)
$$
and thus, in this case, the product of an even distribution by an odd one yields an even result. Hence the product $*$, which is an extension of the dual product, does not preserve some of the symmetry properties of the dual product. Of course, if $F \in \DO$ is odd then $\delta *F=0$. This is not in general the case for $F\in \A$, because when we compute $\delta(x) * F(x)$ we are in fact calculating $\delta(x) * F(x+\epsilon)$, i.e. the product with an infinitesimally {\it shifted} distribution $F_\epsilon(x)=F(x+\epsilon)$ and this, in general, destroys the even/odd symmetry, unless $F$ is smooth. We remark, however, that by using the product $*$ we can easily generate an extension of the dual product that preserves the symmetry properties. For instance, we always have for odd $F \in A$:
$$
\delta(x) * F + F* \delta (x) =0 \, .
$$
in agreement with the standard property.

Other important properties of the product $*$ are summarized in the following theorem:

\begin{theorem} \label{impo}
The product $*$ is an inner operation in $\A$, it is associative, distributive, noncommutative and it reproduces
the H\"ormander product of distributions with disjoint singular supports (and thus the dual product of smooth
functions). In $\A$, the distributional derivative $D$ is an inner operator and satisfies the Leibniz rule with
respect to the product $*$.
\end{theorem}

We conclude that the space $\A$, endowed with the product $*$, is
an associative (but noncommutative) differential algebra of
distributions. We have recently proved \cite{DJP16-2}
that it is essentially the unique differential algebra of
distributions that contains $\C^\infty_p$ and satisfies all the
properties stated in the Schwartz impossibility result
\cite{Sch54}.

\subsection{Delta operators}

Using the product $*$ we can define the following operators:

\begin{definition}\label{2}
Let $x_0\in\RE$ and $n\in\N_0$. The $n$th-order "right shifting delta" is the operator
$$
\widehat \delta _ +^{(n)}({x_0}):\A\longrightarrow\A ; \, \widehat \delta _ + ^{(n)}({x_0}){F} = \delta
^{(n)}({x-x_0})*{ F}
$$
and the $n$th-order "left shifting delta" is the operator
$$
\widehat \delta _ - ^{(n)}({x_0}):\A\longrightarrow\A; \, \widehat\delta _ - ^{(n)}({x_0}){ F} = { F} *\delta
^{(n)}({x-x_0}) \, .
$$
For $n=0$ we write only $\widehat \delta _+({x_0})$ and  $\widehat \delta _-({x_0})$; for $x_0=0$ we write
$\widehat \delta _ + ^{(n)}$ and $\widehat \delta _ - ^{(n)}$.

We also define the operators:
\begin{equation} \label{Gamma1}
\widehat \Gamma^{(n)}(x_0):\A \longrightarrow\A; \, \widehat\Gamma^{(n)}(x_0) { F} =\left[ \widehat \delta_
-^{(n)}(x_0)- \widehat \delta_ +^{(n)} (x_0)\right]{ F}
\end{equation}
For $n=0$ and $x_0=0$ we write $\widehat\Gamma(x_0)$ and $\widehat\Gamma^{(n)}$, respectively.
\end{definition}

Let us also consider the standard operator of "multiplication by
the $n$th-order derivative of a Dirac delta" ($n\in \N_0$)
$$
\widehat\delta^{(n)}(x_0): \C^\infty \longrightarrow \DO'; \, \widehat\delta^{(n)}(x_0) f=
\delta^{(n)}(x-x_0) f
$$
where the product $\delta^{(n)}(x-x_0) f$ is the dual product (\ref{dual})
$$
\langle \delta^{(n)}(x-x_0) f,g \rangle = (-1)^n \frac{\partial^n}{\partial x^n}(f \, g)(x_0) \, , \quad
\forall g \in \DO \, .
$$

Notice that the operators $\widehat \delta _\pm^{(n)}({x_0})$ are extensions of $\widehat \delta ^{(n)}({x_0})$
to the space $\A \supset\C^\infty$:
$$
\widehat \delta _ - ^{(n)}({x_0})f=\widehat \delta _ + ^{(n)}({x_0})f=\widehat \delta ^{(n)}({x_0})f\, , \quad
\forall f\in\C^\infty \, .
$$

Moreover, $\wh \delta_\pm (x_0)$ are the weak operator limits of large classes of sequences of concentrated
smooth potentials. Let us state this property precisely for the simplest case $x_0=0$. For every $\epsilon
> 0$, let $v_\epsilon \in \DO$ be a non-negative, even function such that supp $v_\epsilon \subseteq [-\epsilon
, \epsilon]$  and $ \int_{-\infty}^\infty \, v_\epsilon (x) \, d x =1$. Since
$$
\lim_{\epsilon \downarrow 0} \langle v_{\epsilon},g \rangle =g(0) \, , \quad \forall g \in \DO
$$
we have, in the sense of distributions, $\lim_{\epsilon \downarrow 0}  v_\epsilon (x) = \delta (x) $. Next,
define the operators:
$$
\wh v_{\pm\epsilon}^{(n)}: \A \longrightarrow \A ; \, \wh v_{\pm\epsilon}^{(n)} F(x) = v_\epsilon^{(n)}
(x\mp\epsilon)  F(x) \, ,
$$
where the product {$v_\epsilon^{(n)} (x\mp\epsilon)  F(x)$} is the dual product. In the distributional sense we have, once again, $\lim_{\epsilon
\downarrow 0} v_\epsilon^{(n)} (x \mp \epsilon) = \delta^{(n)} (x) $. On the other hand, in the weak operator
sense:

\begin{theorem}
For all $n \in \N_0$, the one parameter families $\left(\widehat v_{\pm\epsilon}^{(n)}\right)_{(\epsilon >0)}$
converge, in the weak operator sense, to the operators $\widehat\delta_\pm^{(n)}$, i.e.
$$
w-\lim_{\epsilon \downarrow 0} \widehat v_{\pm\epsilon}^{(n)} = \widehat\delta_\pm^{(n)}.
$$
\end{theorem}
The proof is given in [Theorem 3.3, \cite{DPJ16}]. To finish this
section, we prove a simple result that will be used in the sequel.

\begin{theorem}\label{hhh1}
Let ${F} \in \A$. Then, in view of Theorem \ref{18}, there exist
${ f} _-,{ f} _+ \in \C^{\infty}(\RE)$, and $\epsilon
>0$ such that ${ F} |_{(-\epsilon,0)}={f} _- \chi_{(-\epsilon,0)}$ and ${ F} |_{(0,\epsilon)}={ f} _+
\chi_{(0,\epsilon)}$. We then have (for all $j,n \in \N_0$):
\begin{eqnarray}\label{hj}
       \widehat\delta^{(n)} _ \pm D^{j}{F}  = \sum_{k = 0}^n {{( - 1)}^{k+n}} \left( \begin{gathered}
  n \hfill \\
  k \hfill \\
\end{gathered}  \right)\delta^{(k)} (x){D^{n-k+j}{ f} _ \pm }(0)
\, .
\end{eqnarray}
\end{theorem}

\begin{proof}
From Definition \ref{2} and using eq.(\ref{prodf}) we have:
$$
\widehat\delta^{(n)}_- { F} ^{(j)}= { F} ^{(j)} * \delta^{(n)}(x)= \delta^{(n)}(x) { f}^{(j)}_-
$$
and likewise:
$$
\widehat\delta^{(n)}_+ {F}^{(j)}=\delta^{(n)}(x) * {F}^{(j)} =\delta^{(n)}(x) { f}^{(j)}_+
$$
Since ${f}_\pm\in\C^\infty$, a standard result (cf. \cite
[p.36]{Kan98}) states that:
\begin{eqnarray*}
\delta^{(n)}(x) { f}^{(j)}_\pm = \sum_{k = 0}^n {{( - 1)}^{k+n}} \left(
\begin{gathered}
  n \hfill \\
  k \hfill \\
\end{gathered}  \right)\delta^{(k)} (x){D^{n-k+j}{f}_ \pm }(0)
\end{eqnarray*}
which concludes the proof.
 \end{proof}

\section{A new derivative operator} \label{new}

In this section we construct a new derivative operator which is a sort of "covariant derivative" for
discontinuous functions.

\begin{definition}
Let ${I} \subset \RE$ be a finite set. The derivative operator $\widetilde D_I$ is defined by:
\begin{equation}\label{pp}
\widetilde D_{I}:\A\longrightarrow\A; \, \widetilde D_{I} {F}=D {F}+ \sum_{x_0 \in {I}} \widehat\Gamma({x_0}) {F}
\end{equation}
where $\wh\Gamma (x_0)$ was defined in eq.(\ref{Gamma1}).\\
As usual, the higher order derivatives are given by:
$$ \widetilde D^n_{I} {F} = \widetilde D_{I}\left( \widetilde D^{n-1}_{I}
{F}\right) \, , \quad n \in \N \, .
$$
If ${I}=\left\{0\right\}$, we write only $\widetilde D$, instead of $ \widetilde
D_{\left\{0\right\}}$.
\end{definition}

Let us study the main properties of $\widetilde  D_{I}$.

\begin{theorem} \label{PDtilde}
The operator $ \widetilde  D_{I} :\A\rightarrow \A$ is linear, local and satisfies the Leibniz rule with
respect to the product $*$.
\end{theorem}
\begin{proof}
$\widetilde  D_{I}$ is linear because both $D$ and $\wh\Gamma(x_0)$ are linear operators. Likewise,
$ \widetilde D_{I}$  is local because both $D$ and $\wh\Gamma(x_0)$ are local, i.e. supp $D F
\subseteq $ supp $F$ and supp $\wh\Gamma(x_0)F \subseteq $ supp $F$ for all $F \in \A$ (the latter relation
follows easily from (\ref{Gamma1}) and (\ref{hj})).

Next we prove the Leibniz rule for ${I} =\{0\}$; the more general
case is proven exactly in the same way. Let ${F},{G}\in\A$. Then
${F}*{G}\in\A$ and
\begin{eqnarray*}
  \widetilde D ({F}*{G})&=&D ({F}*{G}) +  ({F}*{G}) *
\delta(x)-\delta(x) * ({F}*{G})\\
&=&  (D {F})*{G}+ {F}*(D{G}) +  {F}*({G} * \delta(x))-(\delta(x) * {F})*{G},
\end{eqnarray*}
where we used the fact that $D$ satisfies the Leibniz rule and the product $*$ is associative (cf. Theorem
\ref{impo}). Adding $({F}*\delta(x))*{G}$, subtracting ${F}*(\delta(x)*{G})$ and rearranging the terms in
the previous expression, we get
\begin{eqnarray*}
 {\widetilde D} ({F}*{G}) & = & (D {F})*{G} + ({F}*\delta(x))*{G} - (\delta(x) * {F})*{G} \\
&& + {F}*(D{G}) + {F}*({G}*\delta(x))-{F}*(\delta(x)*{G}) \, .
\end{eqnarray*}
Since the product $*$ is distributive (cf. Theorem \ref{impo}), we get
\begin{eqnarray*}
 {\widetilde D}({F}*{G})&=&(D {F}+{F}*\delta(x) -\delta(x) * {F})*{G}+{F}*(D{G} +  {G}*\delta(x)-\delta(x)*{G})\\
&=&(\widetilde  D{F})*{G}+{F}*  (\widetilde D {G})
\end{eqnarray*}
which concludes the proof.
\end{proof}

Let us now calculate the action of $\wt D_{I}$ for several cases, explicitly. It is trivial to
realize from eq.(\ref{pp}) that:
\begin{equation} \label{DCases}
\wt D_{I} f = D f  \quad \mbox{and} \quad \wt D_{I} \Delta = D \Delta
\end{equation}
for every smooth function $f\in \C^\infty$ and every distribution $\Delta$ with {\it finite} support. More
generally,
\begin{theorem} \label{DFormFTheo}
Let ${I} \subset \RE$ be a finite set and let $F \in \A_{I}$. In view of Theorem
\ref{f6}, $F$ can be written in the form $F=f+\Delta$, where $f=\sum_{i=1}^m f_i \chi_{\Omega_i}$, $f_i \in
\C^\infty (\RE)$ and supp $\Delta \subseteq {I}$. Then{, for every non-negative integer
$k$,} we have:
\begin{equation} \label{DFormF}
\widetilde D^k_{I} F= \sum_{i=1}^m f_i^{(k)} \chi_{\Omega_i} + D^k \Delta \, .
\end{equation}
\end{theorem}

\begin{proof}
We proceed by induction. The identity (\ref{DFormF}) is trivial for $k=0$. Moreover
\begin{eqnarray*}
  \widetilde D^{k+1}_{I} F=\widetilde D_{I} \left(  \widetilde D^{k}_{I} F \right)
\end{eqnarray*}
and using eqs.(\ref{DCases},\ref{DFormF}), we get
\begin{eqnarray*}
  \widetilde D^{k+1}_{I}F = \widetilde D_{I}\left(
\sum_{i=1}^m f_i^{(k)} \chi_{\Omega_i} + D^k \Delta
  \right) = \sum_{i=1}^m \widetilde D_{I} \left( f_i^{(k)} \chi_{\Omega_i} \right) + D^{k+1} \Delta
  \, .
\end{eqnarray*}

Since $\widetilde D_{I}$ satisfies the Leibniz rule with respect to the product $*$ (cf. Theorem
\ref{PDtilde}), and $f_i^{(k)} \chi_{\Omega_i} = f_i^{(k)}
* \chi_{\Omega_i}$ (since $f_i^{(k)}$ is smooth), we have
\begin{eqnarray*}
\widetilde D^{k+1}_{I} F= \sum_{i=1}^m f_i^{(k+1)} \chi_{\Omega_i} + f_i^{(k)} *
\left(\widetilde D_{I}
\chi_{\Omega_i} \right) +D^{k+1}\Delta \, .
\end{eqnarray*}
Finally, from eq.(\ref{pp}) we easily realize that if $\partial\Omega_i \subseteq {I}$ then $\wt
D_{I} \chi_{\Omega_i}=0$. Hence, eq.(\ref{DFormF}) is valid for all $k\in \N_0$.

\end{proof}

The previous result shows that $\wt D_{I}$ (contrary to $D$) is an inner operator in $\C_p^\infty \cap \C^\infty (\RE \backslash {I}) $. More generally,
$$
\wt D_{I}: \, \C_p^\infty \cap \C^\infty (\RE \backslash {J}) \longrightarrow \C_p^\infty \cap \C^\infty (\RE \backslash {J})
$$
for every ${J} \subseteq {I}$.

Another simple consequence of Theorem \ref{DFormFTheo} is:

\begin{corollary}
Let ${F}\in\A_{\left\{0\right\}}$  be written in the form (\ref{0}). Then, for any non negative integer $k$,
\begin{equation} \label{qq}
  \widetilde D^k {F}=H_-D^k{f}_-+HD^k{f}_++D^k\Delta \, ,
\end{equation}
and, in particular, $\widetilde D H=\widetilde D H_- =0$. Recall
that $\widetilde D$ denotes ${\widetilde D} _{\{0\}}$.
\end{corollary}

If $F \notin \A_I$ then the relation between $\widetilde D^k_I {F} $ and $D^k{F}$ is not
so simple. For $I=\{0\}$, we have:

\begin{theorem}
Let ${F}\in\A$. Then for any positive integer $n$,
\begin{equation} \label{importanteJ}
  \widetilde D^{n}{F}=(D+\widehat\Gamma)^n {F}=D^{n} {F}+ \sum_{j=1}^n \left( \begin{gathered}
  n \\
  j  \\
\end{gathered}  \right)
 \widehat\Gamma^{(j-1)} D^{n-j}  {F}
\end{equation}
\end{theorem}

\begin{proof}
We proceed by induction. The case $n=1$ is obvious. Suppose that (\ref{importanteJ}) holds for $n\in\N$. Then
\begin{eqnarray}\label{importantem}
\nonumber  \widetilde  D^{n+1}{F}&=&(D+\widehat\Gamma)^{n+1}{F}=(D+\widehat\Gamma)(D+\widehat\Gamma)^n {F}\\
&=& (D+\widehat\Gamma)(D^{n}{F}+ \sum_{j=1}^n \left( \begin{gathered}
  n  \\
  j  \\
\end{gathered}  \right)
 \widehat\Gamma^{(j-1)} D^{n-j}  {F})\\
\nonumber &=& D^{n+1}{F}+ \widehat\Gamma D^n {F} + \sum_{j=1}^n \left( \begin{gathered}
  n \\
  j  \\
\end{gathered}  \right)
D\widehat\Gamma^{(j-1)} D^{n-j}  {F}
\end{eqnarray}
where in the last step we took into account that $\widehat\Gamma \left( \widehat\Gamma^{(j-1)}D^{n-j}{F}\right)=0$. \\
Noticing that
\begin{eqnarray}\label{importanten}
 D\widehat\Gamma^{(j-1)} D^{n-j}  {F}=\widehat\Gamma^{(j)}D^{n-j}  {F}+\widehat\Gamma^{(j-1)}D^{n-j+1}  {F}
\end{eqnarray}
and substituting (\ref{importanten}) in equation (\ref{importantem}) we get

\begin{eqnarray*}
   \widetilde D^{n+1}{F}&=&D^{n+1}{F}+ \sum_{j=0}^n \left( \begin{gathered}
  n  \\
  j  \\
\end{gathered}  \right)
 \widehat\Gamma^{(j)} D^{n-j}  {F}+ \sum_{j=1}^n \left( \begin{gathered}
  n  \\
  j  \\
\end{gathered}  \right)
 \widehat\Gamma^{(j-1)} D^{n-j+1}  {F}\\
&=&D^{n+1}{F}+ \sum_{j=1}^{n+1} \left( \begin{gathered}
  n  \\
  j-1  \\
\end{gathered}  \right)
 \widehat\Gamma^{(j-1)} D^{n-j+1}  {F}+ \sum_{j=1}^n \left( \begin{gathered}
  n  \\
  j \\
\end{gathered}  \right)
 \widehat\Gamma^{(j-1)} D^{n-j+1}  {F}\\
&=& D^{n+1}{F}+ \sum_{j=1}^{n+1} \left( \begin{gathered}
  n+1  \\
  j  \\
\end{gathered}  \right)
 \widehat\Gamma^{(j-1)} D^{n-j+1}  {F}
\end{eqnarray*}
where in the last equality we used Pascal's identity. Hence, eq.(\ref{importanteJ}) is valid for all $n\in \N$.
\end{proof}

\section{Linear Differential Equations with point interactions} \label{Linear}

We now return to Problem 1. Recall that ODE$_1$ is a linear ordinary differential equation of the form:
\begin{equation} \label{iiii}
\sum_{i=0}^n a_i D^{i}\psi=f
\end{equation}
where $a_i,f\in\C^{\infty}(\RE)$ and ${a_n}(x) \ne 0$, for all $x
\in \mathbb{R}$. It follows from Picard's theorem and other
general results in the theory of linear ODEs [Lemma 2.3 and
Theorem 3.9 \cite{Teschl}] that a solution of (\ref{iiii}) exists,
is smooth and globally defined on $\RE$ (and is also unique for
each initial data).

The aim of this section is to derive a new differential equation
(denoted ODE$_2$) whose solutions satisfy the conditions (C1) and
(C2) stated in Problem 1. To simplify the presentation we will
consider the particular case ${I}=\{0 \}$. The general case where
${I}$ is an arbitrary finite set, can be solved exactly in the
same way. For ${I} =\{0\}$, the conditions (C2) can be written
as
\begin{equation} \label{ICS}
A \overline{\psi(0^-)}= B \overline{\psi(0^+)}
\end{equation}
where $A=[A_{ij}]$ and $B=[B_{ij}]$ are two $m \times n$ matrices with $m \le n$, and $0 \le i \le m-1$, $0 \le
j \le n-1$ (the lower bound of $i,j$ is set to zero in order to simplify the presentation).

The main results of this section are given in Theorem \ref{importante} and Theorem \ref{importanteq}. Some
preparatory results are presented in Theorem \ref{F1F2} and Corollary \ref{4.3}. We start by defining an
operator that will be used to impose the conditions (C2).

\begin{definition}
Let $A,B$ be defined as above. The "interface operator" $\widehat F:\A\rightarrow\DO'$ is the singular, rank $m$
operator defined by
\begin{equation} \label{bf}
\widehat F=\sum_{i=0}^{m-1} \sum_{j=0}^{n-1} {(A_{ij} D^i \wh\delta_- D^j-B_{ij} D^i \wh\delta_+ D^j )}\, .
\end{equation}
\end{definition}

The next Theorem provides two equivalent formulations of $\wh F$.

\begin{theorem} \label{F1F2}
Let $\psi \in \A$. The operator (\ref{bf}) can be written as:
\begin{equation}\label{F1}
 \widehat F\psi=\sum_{j=0}^{n-1}\sum_{i\geq k=0}^{m-1} \left( \begin{gathered}
  i \\
  k  \\
\end{gathered}  \right) \left({ A_{ij}\widehat \delta_-^{(k)}-B_{ij}\widehat \delta_+^{(k)}}\right) D^{i+j-k}\psi
\end{equation}
and also as:
\begin{equation}\label{F2}
 \widehat F\psi= \sum_{i=0}^{m-1}\sum_{j=0}^{n-1} \left({A_{ij}\psi_-^{(j)}(0) -B_{ij}\psi_+^{(j)}(0)}\right) \delta^{(i)}(x)
\end{equation}
where the functions $\psi_\pm \in \C^\infty(\RE)$ (associated with $\psi$) are given explicitly by Theorem
\ref{18}.

\end{theorem}

\begin{proof}
From the definitions of $\wh F$ and $\wh\delta_\pm$, we have:
\begin{equation} \label{F3}
\wh F \psi =\sum_{i=0}^{m-1} \sum_{j=0}^{n-1} {\left(  A_{ij} D^i (\psi^{(j)} *
\delta (x) )-B_{ij} D^i (\delta(x) *\psi^{(j)}) \right)}
\end{equation}
Since $D$ satisfies the Leibniz rule with respect to $*$ (cf. Theorem \ref{impo}),
\begin{equation*}
D^i (\delta (x) * \psi^{(j)})  =  \sum_{k=0}^{i} \left( \begin{gathered}
  i \\
  k  \\
\end{gathered}  \right)
\delta^{(k)}(x) * \psi^{(i+j-k)}   = \sum_{k=0}^{i} \left( \begin{gathered}
  i \\
  k  \\
\end{gathered}  \right)
\wh\delta_+^{(k)} D^{i+j-k} \psi
\end{equation*}
and likewise:
$$
D^i(  \psi^{(j)}* \delta(x)) = \sum_{k=0}^{i} \left( \begin{gathered}
  i \\
  k  \\
\end{gathered}  \right)
\wh\delta_-^{(k)} D^{i+j-k} \psi
$$
Substituting these formulas in (\ref{F3}), we obtain (\ref{F1}).

We now go back to (\ref{F3}). From eq.(\ref{hj}) we have:
\begin{equation} \label{F4aux}
\widehat \delta _\pm \psi^{(j)}=\delta (x) \psi^{(j)}_\pm(0), \quad  j=0,\ldots , n-1
\end{equation}
where, for each $\psi \in \A$, the functions $\psi_-,\psi_+ \in \C^\infty(\RE)$ are such that, for some
$\epsilon
> 0$
$$
\psi |_{(-\epsilon,0)} = \psi_- \chi_{(-\epsilon,0)} \quad \mbox{and} \quad \psi |_{(0,\epsilon)} = \psi_+
\chi_{(0,\epsilon)} \, .
$$
These functions always exist (cf. Theorem \ref{18}). Substituting (\ref{F4aux}) in (\ref{F3}) we get (\ref{F2}),
concluding the proof.

\end{proof}

An important corollary of the previous result is:

\begin{corollary} \label{4.3}
Let $\widehat F$ be the operator (\ref{bf}). Then $\widehat F$ satisfies:
\begin{itemize}
\item [(i)]$\text{supp}(\widehat F \psi) \subseteq \left\{0\right\}$, for all $\psi\in \A $
\item [(ii)] $\text{Ker}(\widehat F )=\left\{\psi\in \A: A\overline{\psi_-({0})} =B\overline{\psi_+({0 })}\right\}$,
\end{itemize}
\end{corollary}

\begin{proof}
It follows from (\ref{F2}) that $\text{supp}( \widehat F\psi) \subseteq \left\{0\right\}$ for all $\psi\in \A$.
Moreover, also from (\ref{F2})
\begin{eqnarray*}
\psi\in \text{Ker}(\widehat F )&\Leftrightarrow & \sum_{i=0}^{m-1} \delta^{(i)}(x) \left( \sum_{j=0}^{n-1}
A_{ij}\psi_-^{(j)}(0)-B_{ij}\psi_+^{(j)}(0) \right)=0\\
&\Leftrightarrow& A\overline{\psi_-({0})} =B\overline{\psi_+({0 })}
\end{eqnarray*}
 and so $\widehat F$ also satisfies the condition (ii).
\end{proof}
\indent

The next Theorems \ref{importante} and {\ref{importanteq}} provide two equivalent formulations of the ODE$_2$.

\begin{theorem}\label{importante}
Let the ODE$_1$ be given by eq.(\ref{iiii}), and the interface conditions by (\ref{ICS}). Then the ODE$_2$ can
be written explicitly as:
\begin{equation}\label{importanteC}
\sum_{i=0}^n a_{i}   \widetilde D^{i} \psi + \widehat F \psi=f
\end{equation}
where $\widetilde D$ is the new derivative operator (\ref{pp}) with $I=\{0\}$, and $\widehat F$ is given by eq.(\ref{bf}).
\end{theorem}

\begin{proof}
For a general $\psi\in\A$ the support of $\wh F \psi$ is, at the most, $\left\{0\right\}$ (cf. Corollary
\ref{4.3}). Moreover $\widetilde D{^i} \psi = D{^i} \psi,{\, i=0,...,n}$ on
$\RE \backslash \{0\}$ (cf. eq.(\ref{importanteJ})). Hence, on $\RE_{\mathbin{+}}$ and
$\RE_{\mathbin{-}}$, eq.(\ref{importanteC}) reduces to eq.(\ref{iiii}) and thus any global solution
of eq.(\ref{importanteC}) will be of the form
\begin{equation}\label{importanteE}
\psi=H_-\psi_-+H\psi_++\Delta,
\end{equation}
where $\psi_-,\psi_+ \in \C^{\infty}(\RE)$ and supp $\Delta\subseteq\left\{0\right\}$. Notice that $\psi$ in
eq.(\ref{importanteE}) is the most general distribution $\psi \in \DO'$, such that $\psi|_{\RE_-}$ and
$\psi|_{\RE_+}$ are smooth functions with smooth extensions to $\RE$ (recall that the solutions of ODE$_1$ are
smooth and maximally defined on $\RE$).

Substituting (\ref{importanteE}) in (\ref{importanteC}), and using (\ref{qq}) and (\ref{F2}), we get from (\ref{importanteC}): 
\begin{eqnarray}\label{importantee}
&&\sum_{i=0}^n a_{i}  (H_-D^{i}\psi_-+HD^{i}\psi_+) + \sum_{i=0}^n a_{i}  D^{i}\Delta \\
\nonumber &&+\sum_{i=0}^{m-1}\sum_{j=0}^{n-1}\delta^{(i)}(x) \left({A_{ij}\psi_-^{(j)}(0)
-B_{ij}\psi^{(j)}_+(0)} \right) =f \, .
\end{eqnarray}
Separating the terms that involve the Dirac delta and its derivatives from those that do not, we obtain
$$\sum_{i=0}^n a_{i}  (H_-D^{i}\psi_-+HD^{i}\psi_+)=f
\Leftrightarrow \left\{ \begin{array}{rll}
  \sum\limits_{i=0}^n a_{i} D^{i}\psi_- = f&\hbox{on}& \RE_{\mathbin{- }}  \\
  \sum\limits_{i=0}^n a_{i} D^{i}\psi_+ = f&\hbox{on}& \RE_{\mathbin{+}}
\end{array}\right.$$
and thus, as required, $\psi$ satisfies eq.(\ref{iiii}) on $\RE \backslash \{0\}$. Moreover,
\begin{equation}\label{imm}
 \sum\limits_{i=0}^n a_{i} D^{i}\Delta +E = 0
\end{equation}
where
$$
E= \wh F \psi=\sum_{i=0}^{m-1}\sum_{j=0}^{n-1}\delta^{(i)}(x)\left({A_{ij}\psi_-^{(j)}(0) -B_{ij}\psi^{(j)}_+(0)} 
\right) \, .
$$
Since $\text{supp }\Delta\subseteq\left\{0\right\}$, we have $\Delta=0$ or ord $(\Delta )\geq1$. In the latter
case, ord $(D^{n}\Delta )\geq n+1$ and so ord $(a_n D^{n}\Delta) \geq n+1$ (recall that $a_n(x) \neq 0$,
$\forall x \in \RE$). Moreover, we also have ord $E\leq m$. Taking into account that the terms of different
orders in (\ref{imm}) are linearly independent, and that $m\leq n$, we conclude that $a_nD^{n}\Delta$ cannot be
cancelled by any other term in (\ref{imm}). Hence $\Delta=0$, and thus equation (\ref{imm}) reduces to
\begin{eqnarray} \label{BCond}
E=0 &\Longleftrightarrow & \sum_{i=0}^{m-1}\sum_{j=0}^{n-1}\left({A_{ij}\psi_-^{(j)}(0) -B_{ij}\psi^{(j)}_+(0)}\right)\delta^{(i)}(x)=0 \nonumber\\
&\Longleftrightarrow & \sum_{j=0}^{n-1} {A_{ij}\psi_-^{(j)}(0)}=\sum_{j = 0}^{n - 1} {B_{ij}\psi^{(j)}_+(0)}\, ,
\quad i = 0, \ldots ,m-1
\end{eqnarray}
which are just the interface conditions (C2), given in (\ref{ICS}).

Assembling all these results, we conclude that the global solutions of (\ref{importanteC}) are of the form
$\psi=H_-\psi_-+H\psi_+$, where $\psi_-,\psi_+$ satisfy (\ref{iiii}) and the interface conditions (\ref{BCond})
at $x=0$. Hence, eq.(\ref{importanteC}) yields a possible formulation of ODE$_2$, which concludes the proof.

\end{proof}

\begin{remark} \label{4.5}

Notice that the ODE$_2$ is not uniquely defined for each ODE$_1$, and for each interface condition. In fact, any
other operator $\wh F$ that satisfies the two properties in Corollary \ref{4.3} and such that ord $\wh F \psi
\le n$ for all $\psi \in \A$, can equally well be used to define the ODE$_2$. We can easily conclude that this
is true by reviewing the role played by $\wh F$ in the proof of the previous Theorem.
\end{remark}

The following are two interesting particular cases of eq.(\ref{importanteC}):

\begin{corollary} \label{Corollary4.6}
If we set $A=B=0$ in eq.(\ref{ICS}) (that is, if there are no interface conditions at $x=0$) then $\wh
F=0$, and the ODE$_2$ reduces to:
\begin{equation}
\sum_{i=0}^n a_{i}   \widetilde D^{i} \psi =f \, .
\end{equation}
Its solutions are of the form $\psi=H_- \psi_- + H \psi_+$, where $\psi_\pm$ are two (independent) solutions of
(\ref{iiii}). Hence, $\psi$ might display an arbitrary discontinuity at $x=0$ and thus the dimension of the space of
solutions of ODE$_2$, in this case, is twice the dimension of the space of solutions of ODE$_1$.

If, on the other hand, $A=B ={\bf 1}_{n \times n}$ (where ${\bf 1}_{n \times n}$ is the $n \times n$ identity
matrix) then the ODE$_2$ is equivalent to ODE$_1$ (\ref{iiii}).

\end{corollary}

We now present an alternative formulation of the ODE$_2$, which is written in terms of the
standard distributional derivative $D$, the product $*$, and an additive (singular) perturbation of the coefficients:

\begin{theorem} \label{importanteq}
The ODE$_2$ (\ref{importanteC}) can equivalently be written as:
\begin{equation}\label{importanteI}
\sum_{i=0}^n\left(  \widetilde a_i * \psi^{(i)}+ \psi^{(i)} *  \widetilde b_i\right) + \widehat F \psi=f
\end{equation}
where $\widehat F$ is the operator {(\ref{bf})}, $\widetilde
a_i,\widetilde b_i$ are given by:
\begin{eqnarray}\label{hh}
 \widetilde a_i=\frac{1}{2}a_{i}-\sum_{k=1}^{n-i} a_{i+k} \left(\begin{gathered}
 i+k \\
  k  \\
\end{gathered}  \right)\delta^{(k-1)}(x), \quad i=0,...,n-1
\end{eqnarray}
\begin{eqnarray}\label{hhh}
  \widetilde b_i=\frac{1}{2}a_{i}+\sum_{k=1}^{n-i} a_{i+k} \left(\begin{gathered}
 i+k \\
  k  \\
\end{gathered}  \right)\delta^{(k-1)}(x), \quad i=0,...,n-1
\end{eqnarray}
and $\widetilde a_n=\widetilde b_n = a_n/2$. Notice that $\widetilde a_i+\widetilde b_i = a_i$ for all
$i=0,...,n$.
\end{theorem}

\begin{proof}
{ We want to show that the equations (\ref{importanteC}) and (\ref{importanteI}) are equivalent. Substituting
(\ref{importanteJ}) in equation (\ref{importanteC}) we get:}
\begin{eqnarray}\label{kk}
{\mathbin{\sum_{i=0}^n a_{i} \psi^{(i)}+ \sum_{i\geq j=1}^n  a_{i} \left( \begin{gathered}
  i \hfill \\
  j \hfill \\
\end{gathered}  \right) \widehat \Gamma^{(j-1)}\psi^{(i-j)}+\widehat F\psi=f}}
\end{eqnarray}
The first term, in (\ref{kk}), can be written as
 \begin{eqnarray}\label{kk1}
 {\mathbin{\sum_{i=0}^n a_{i} \psi^{(i)}=\sum_{i=0}^n  \left(\frac{a_{i}}{2} *\psi^{(i)}+\psi^{(i)}*\frac{a_{i}}{2}\right)}}
 \end{eqnarray}
since the coefficients $a_i$ are smooth. Moreover, the second term satisfies:
\begin{eqnarray}\label{kk2}
{\mathbin{\sum_{i\geq j=1}^n  a_{i} \left( \begin{gathered}
  i \hfill \\
  j \hfill \\
\end{gathered}  \right) \widehat \Gamma^{(j-1)}\psi^{(i-j)}=\sum_{i=0}^{n-1} \sum_{k=1}^{n-i} a_{i+k} \left( \begin{gathered}
  i+k  \\
  k  \\
\end{gathered}  \right) \widehat \Gamma^{(k-1)}\psi^{(i)}}} \, .
\end{eqnarray}
Substituting (\ref{kk1}) and (\ref{kk2}) in (\ref{kk}) and using
(\ref{Gamma1}), we easily obtain (\ref{importanteI}) with $
\widetilde a_i$ and $\widetilde b_i$ defined by (\ref{hh}) and
(\ref{hhh}).
\end{proof}

\section{Simple example}\label{Simple example}

In this section we study a simple example in order to illustrate
the previous results. Let the ODE$_1$ be the equation:
\begin{equation}\label{yy}
\psi''+k^2\psi=0,
\end{equation}
where $k\in \RE_+$. Assume that ${I} =\{0\}$, and that the interface conditions (C2)
are of the form:
\begin{equation}\label{jji}
{ A} \overline{\psi(0^-)}={ B} \overline{\psi(0^+)}
\end{equation}
where:
\begin{equation}\label{Matricesex}
\ A= \left[ \begin{array}{rr} k_1 & 0 \\
0 & k_2
\end{array} \right] \quad , \quad
\ B= \left[ \begin{array}{cc} 1 & 0\\
0 & 1
\end{array} \right]
\end{equation}
with $\ k_1,k_2\in \RE$. Notice that if $\ k_1=k_2=0$ then the
conditions (\ref{jji}) are separating. If $\ k_1,k_2\neq0$,
they are interacting.

According to Theorem \ref{importante}, the ODE$_2$ for this system can be written in the form
\begin{equation}\label{hh1}
  \widetilde D^{2} \psi +k^2\psi+ \widehat F \psi=0,
\end{equation}
where  $ \widetilde D$ is the new derivative operator (\ref{pp})
and $\widehat F$ is given by (\ref{bf}) (notice that $m=n=2$):
\begin{equation}\label{r}
\widehat F\psi=(k_1 \widehat \delta _--\widehat \delta _+)\psi +
(k_2\widehat \delta' _--\widehat \delta' _+ )\psi'+(k_2\widehat
\delta _--\widehat \delta _+ )\psi'' \, .
\end{equation}

Let us then solve (\ref{hh1}) explicitly. From (\ref{importanteJ})
and (\ref{r}), we easily conclude that on $\RE_{\mathbin{+}}$ and
$\RE_{\mathbin{-}}$ the eq.(\ref{hh1}) reduces to (\ref{yy}).
Hence, its global solution is of the form
\begin{equation}\label{eE}
\psi=H_-\psi_-+H\psi_++\Delta,
\end{equation}
where $ \text{supp} \, \Delta\subseteq\left\{0\right\}$  and
$\psi_\pm \in \C^{\infty}$ satisfy (\ref{yy}). Moreover, from
(\ref{F2}),
\begin{equation}\label{rr}
\widehat F\psi= \delta(x)(k_1 \psi_-(0)-\psi_+(0)) +\delta '(x)( k_2 \psi'_-(0)-\psi'_+(0)) \, .
\end{equation}
Substituting (\ref{eE}) in (\ref{hh1}), and using (\ref{qq}) and (\ref{rr}), we get
\begin{eqnarray}\label{cc}
&&H_-(\psi''_-+k^2\psi_-)+H(\psi''_++k^2\psi_+)+\Delta''+k^2 \Delta\\
\nonumber \ && +\delta(x)(k_1 \psi_-(0)-\psi_+(0))+\delta '(x)( k_2 \psi'_-(0)-\psi'_+(0))=0
\end{eqnarray}
Separating the terms that depend on the delta Dirac from those that do not, we get from (\ref{cc})\\
$$
H_-(\psi''_-+k^2\psi_-)+H(\psi''_++k^2\psi_+)=0 \Longleftrightarrow \left\{ \begin{array}{l}
\psi''_-+k^2\psi_-=0 \quad \mbox{on}\quad \RE_- \\
\psi''_++k^2\psi_+=0 \quad \mbox{on}\quad \RE_+
\end{array}
\right.
$$
and
\begin{equation}\label{im}
\Delta''+k^2 \Delta+\delta(x)(k_1 \psi_-(0)-\psi_+(0)) +\delta '(x)( k_2 \psi'_-(0)-\psi'_+(0))= 0
\end{equation}
The terms of order higher than two yield $\Delta'' + k^2 \Delta=0$. Since supp $\Delta \subseteq \{0\}$ this
implies $\Delta=0$. Hence, eq.(\ref{im}) reduces to
$$\delta(x)(k_1 \psi_-(0)-\psi_+(0)) +\delta '(x)( k_2 \psi'_-(0)-\psi'_+(0)) = 0
\Leftrightarrow \left\{ \begin{array}{rll}
   \psi_+(0)=k_1\psi_-(0) \\
   \psi'_+(0)=k_2\psi'_-(0)
\end{array}\right.$$
which are just the original interface conditions. We conclude that the solutions of ODE$_2$ are of the form
$$
\psi=H_-\psi_-+H\psi_+
$$
where $\psi_\pm$ satisfy the eq.(\ref{yy}) and the interface
conditions (\ref{jji},\ref{Matricesex}).

Equivalently, the ODE$_2$ can be written in the form given by Theorem \ref{importanteq}
\begin{equation} \label{SformEx1}
\sum_{i=0}^2\left(  \widetilde a_i * \psi^{(i)}+ \psi^{(i)} *  \widetilde b_i\right) + \widehat F \psi=0
\end{equation}
where $\widehat F$ is  given by (\ref{r}) or (\ref{rr}), and
$$
\left\{ \begin{array}{l} \widetilde a_0=\tfrac{1}{2} k^2-\delta ^{(1)}(x)\\
\widetilde b_0=\tfrac{1}{2} k^2+\delta ^{(1)}(x)
\end{array} \right.
\quad , \quad \left\{ \begin{array}{l} \widetilde a_1=- \widetilde b_1=-2 \delta(x) \\
\widetilde a_2= \widetilde b_2= \tfrac{1}{2}
\end{array} \right.
$$
These coefficients can be easily calculated from (\ref{hh},\ref{hhh}). Notice that $\widetilde a_i+\widetilde
b_i=a_i$, $i=0,1,2$, where $a_i$, $i=0,1,2$ are the coefficients of eq.(\ref{yy}) written in the form
(\ref{iiii}).

Two interesting cases described by eq.(\ref{hh1}) (or by eq.(\ref{SformEx1}))
are the total confining case ($k_1=k_2=0$) where the global
solutions are of the form $\psi= H_- \psi_-$ (and thus {\it
confined} to $\RE_-$, but unrestricted otherwise), and the {\it continuous} conditions,
corresponding to $k_1=k_2=1$, for which the ODE$_2$ is equivalent
to the original ODE$_1$.

\section{The $n$-order derivative operator with point interactions} \label{Derivative}


As a by-product of the previous results, in this section we obtain an explicit relation between the 
singular perturbations of the $n$-th order derivative operator $\wh L_0= (i D )^n$ with domain in the Sobolev space $D(\wh L_0)= \W^n_2(\RE)$, and the extensions (self-adjoint or not) of the symmetric restriction of $\wh L_0$ to $\DO(\RE \backslash \{0\})$:
$$
\wh S = \left. \wh L_0 \right|_{\DO(\RE \backslash \{0\})}
$$
More precisely, we will consider the extensions $\wh L_{AB}$ of $\wh S$, with domain (cf. eq.(\ref{ICS})):
\begin{equation} \label{DomainLAB}
D(\wh L_{AB})= \{ \psi \in \W^n_2(\RE_-) \oplus \W^n_2(\RE_+): 
A \overline{\psi(0^-)}= B \overline{\psi(0^+)} \}
\end{equation}
where $A,B$ are two $m \times n$ matrices ($m \le n$). These operators are restrictions of the adjoint of $\wh S$:
$$
\wh S^*: \W^n_2(\RE_-) \oplus \W^n_2(\RE_+) \longrightarrow \L^2(\RE) 
$$
$$
\wh S^* (H_-\psi_-+H\psi_+) = (i)^n \left( H_- \psi_-^{(n)} + H \psi_+^{(n)} \right)
$$
where $\psi_\pm \in \W^n_2(\RE)$. 
We will then show that each operator $\wh L_{AB}$ coincides with a particular singular perturbation of the operator $\wh L_0$. These singular perturbations are formulated intrinsically using the formalism of the previous sections. 
We remark that the set of extensions of $\wh S$ with domain (\ref{DomainLAB}) includes all self-adjoint extensions of $\wh S$, as well as non-self-adjoint ones. 

In order to present the main result of this section, we first need to slightly extend the domain of definition of the product $*$ and of the operators $\wh\delta^{(n)}_\pm $ and $\wh\Gamma^{(n)}$. Let us introduce the spaces 
$$
\A^n=\C_p^n \oplus \F^n \quad , \quad \A^n_I=(\C_p^n\cap \C^n(\RE \backslash I)) \oplus \F^n_I  \quad , \quad  n \ge 0
$$
where $I\subset \RE$ is a finite set, $\C_p^n$ is the set of piecewise $n$-th order differentiable functions, $\F^n$ (respectively, $\F^n_I$) is the space of Schwartz distributions of finite support (respectively, of support $I$) and of order less or equal to $n+1$. We have of course, $\A^\infty=\A$ and $\A^\infty_I=\A_I$.
The definition of the product $F*G$, given by (\ref{55}), can be trivially extended to $F,G \in \A^n$, and the product still satisfies (\ref{prodf}). Notice that for $F,G \in \A^n$, we have $f_i,g_i \in \C^n$ in (\ref{prodf}). This definition and its properties were studied in detail in \cite{DPJ16}.
   
Using the new product, the operators $\wh\delta^{(n)}_\pm $ and $\wh\Gamma^{(n)}$ can also be extended to $\A^m$ for all $m\ge n$. They still satisfy (\ref{hj}) for $F \in \A^m$ and $m \ge n+j$. Likewise, the operator $\wt D_I$ can be extended to the form $\wt D_I : \A^n \longrightarrow \DO'$ for all $n \ge 0$. This generalization also satisfies (\ref{DFormF}) for $F \in \A_I^{k-1}$, and (\ref{qq}) for $F \in \A_{\{0\}}^{k-1}$. Finally the domain of the interface operator $\wh F$ (\ref{bf}) can be extended to $\A^{n-1}$ and the result (\ref{F2}) is still valid for $\psi \in \A^{n-1}$. Then:

\begin{theorem}

The operator $\wh L_{AB}$, restriction of $\wh S^*$ to the domain (\ref{DomainLAB}), coincides with the operator 
$$
\wh L_F= (i \tilde D)^n + \wh F
$$
defined on its maximal domain
$$
D_{\rm max}(\wh L_F) := \{\psi \in \L^2(\RE) : \, \wh L_F \psi \in \L^2(\RE) \}  \, .
$$
Here, $\wh F$  is the extension of the interface operator (\ref{bf}) to $\A^{n-1}$.

\end{theorem}
 
 \begin{proof}

For $x\not= 0 $ the operator $\wh L_F$ is given by $ (iD)^n $. The maximal domain of $(i D)^n$ for $x \not= 0$ is the Sobolev space $\W^n_2(\RE \backslash \{0\})$.   
Hence, every $\psi \in D_{\rm max}(\wh L_F)$ can be written in the form:
$$
\psi = H_- \psi_- + H \psi_+
$$
where $\psi_-,\psi_+ \in \W^n_2(\RE)$. 

We now prove that $\psi$ satisfies the boundary conditions (\ref{DomainLAB}). Acting with $\wh L_F$, and using (20) we find:
$$
\wh L_F \psi = (i\wt D)^n \psi+\wh F \psi = (i)^n \left( H_-\psi_-^{(n)}+H\psi_+^{(n)} \right) + \wh F \psi
$$
where we took into account that $\psi \in \W^n_2(\RE \backslash \{0\}) \subset \A^{n-1}_{\{0\}}$ and that (\ref{qq}) is valid in this domain. Using (\ref{F2}) (which is also valid for $\psi \in \A^{n-1}_{\{0\}}$) we find:
\begin{eqnarray*}
	\wh L_F \psi \in \L^2 &\Longleftrightarrow & \wh F \psi \in \L^2 \Longleftrightarrow \wh F\psi =0 \\
	&\Longleftrightarrow & \sum_{j=0}^{n-1} {A_{ij}\psi_-^{(j)}(0)}=\sum_{j = 0}^{n - 1} {B_{ij}\psi^{(j)}_+(0)}\, ,
\quad i = 0, \ldots ,m-1
\end{eqnarray*}
which are just the boundary conditions on the domain of $\wh L_{AB}$.

 \end{proof}

An important example from the set of operators $\wh L_F$ is
$$
\wh L_F = \wh S^* = (i\wt D)^n
$$
which corresponds to the case where $A=B=0$.

Finally, we remark that Theorem 6.1 is equally valid if $\wh F:D(\wh F) \to \DO'$  is an arbitrary operator satisfying: 

(i) supp $\wh F \psi \subseteq \{0\}$, for all $\psi \in \W^n_2(\RE \backslash \{0\}) \subseteq D(\wh F)$, 

(ii) Ker $\wh F \cap \W^n_2(\RE \backslash \{0\}) = D(\wh L_{AB})$. \\
and thus there is a large class of operators $\wh L_F$ coinciding with each s.a. perturbation of $\wh L_0$.

\section{Previous approaches and outlook}  \label{Conclusions} 

In this paper we used the intrinsic formalism developed in \cite{DP09,DPJ16} to obtain a solution for the problem of constructing an ODE (strictly defined within the space of Schwartz distributions) whose global solutions satisfy a prescribed set of interface conditions. We addressed the (inverse) problem of constructing an equation from the properties of its global solutions. On the other hand, we have not solved the general problem of determining the (properties of the) solutions of the ODE$_2$ when the coefficients are arbitrary elements of $\A$. This will be the topic of a future work, where we will show, in particular, that in such general case the solutions might be singular distributions, and not only regular functions as in the case considered here. We will also address the problem of existence and uniqueness of the solutions of ODEs with arbitrary coefficients in $\A$.

In this paper we have also studied the closely related problem of determining the singular perturbation of the $n$-th order derivative operator whose domain satisfies a prescribed set of interface conditions. The two problems (for ODEs and operators) were solved in terms of exactly the same type of singular perturbations. The results for the operator $(iD)^n$ generalize those of \cite{DPJ16} where we have already constructed an intrinsic boundary potential formulation for all s.a. Schr\"odinger operators with a point interaction.

The problem of determining explicit formulations of singular perturbations of linear operators was studied before by other authors. Most relevant are the results of Kurasov and Boman \cite{Kurasov,KurasovBoman} using a theory of distributions for discontinuous test functions (see also \cite{Albeverio2} and \cite{Brasche}). We will finish this section with a brief discussion of the relation between their approach and ours.

The space $\K$ of discontinuous test functions is the subspace of $\C^\infty(\RE\backslash \{0\}) \cap \C_p^\infty (\RE)$ of functions with compact support. 
Distributions over discontinuous test functions are the linear and continuous maps from $\K$ to $\CO$, where convergence in $\K$ is defined in the usual sense of test functions. 

In $\K'$ a distributional derivative can be defined by:
$$
\langle D_K\phi , g \rangle =- \langle \phi, \partial_x g \rangle \quad , \quad g \in \K
$$
where $\partial_x$ is the usual pointwise derivative, and $\langle \, ,\, \rangle$ is the dual bracket in $\K' \times \K$.
In the same way, a product of $\phi \in \K'$ by $f \in \K_{\rm loc}=\C^\infty(\RE\backslash \{0\}) \cap \C_p^\infty (\RE)$  can be defined by duality
$$
\langle \phi \star_K f , g \rangle =- \langle \phi, f g \rangle \quad , \quad g \in \K
$$
since $fg \in \K$.      

This formalism was used in \cite{Kurasov} to define singular perturbations of one-dimensional Schr\"odinger operators:
\begin{equation} \label{HK}
\wh H=-D_K^2 + \wh B_H
\end{equation}
where $\wh B_H$ is a singular boundary operator of the form:
$$
\wh B_H \psi = a D_K^2 \delta \star_K \psi + D_K (b \delta + c \delta') \star_K \psi + (d \delta+ e \delta') \star_K \psi  
$$
where $a,b,c,d,e \in \CO$.
The same formalism was also used to defined singular perturbations of the $n$-th order derivative operator \cite{KurasovBoman}
\begin{equation} \label{LK}
\wh L= (i D_K)^n + \wh B_L 
\end{equation}
in terms of a singular interaction term:
$$
\wh B_L= \sum_{i,j=0}^{m-1} c_{ij} \langle \delta^{(j)}, \cdot \, \rangle  \, \delta^{(i)} \, .
$$
Both operators $\wh H$ and $\wh L$ are of the general form (in the former case $n=2$):
$$
\wh A :D(\wh A) \subset \W^n_2(\RE \backslash \{0\}) \longrightarrow \K'
$$
and if we set $D(\wh A)=D_{\rm max}(\wh A)=\{\psi \in \W^n_2(\RE\backslash \{0\}) : \wh A \psi \in \L^2(\RE) \}$ then $\wh A$ is {\it not} s.a. \cite{Kurasov}. In order to produce physically sensible results, we have to compose $\wh A$ with the projector 
$$
\wh\eta : \K' \longrightarrow \DO'
$$
to get $\wh A_D=\wh\eta\wh A$. Then, for suitable boundary operators $\wh B_H$ and $\wh B_L$, the operator $\wh A_D$ in the domain:
$$
D(\wh A_D) = D_{\rm max}(\wh A_D)=\{\psi \in \W^n_2(\RE\backslash \{0\}): \, \wh A_D \psi \in \L^2(\RE) \} \, ,
$$
coincides with one of the s.a. perturbations of $\wh H_0$ or $\wh L_0$, respectively.

Using the operators $\wh\eta \wh H$, Kurasov was able to determine a singular perturbation formulation for the entire set of s.a. Schr\"odinger operators with a point interaction \cite{Kurasov}. 
An equivalent result for the $n$-th order derivative operator was obtained by Kurasov and Boman using the operators $\wh\eta \wh L$ \cite{KurasovBoman}. 

In the context of differential operators, the motivation of Kurasov's formulation is, of course, closely related to ours. The two formalisms, however, display some important differences: 

(1) Kurasov's formulation of the operators with singular perturbations, like the formulations in terms of generalized functions, is not intrinsic. The new operators are written in terms of the new distributions in $\K'$, the new derivative operator $D_K$, and a new product $\star_K$. This structure satisfies a delicate set of properties: $D_K$ does not satisfy the Leibniz rule, the derivative of a constant is not zero, and even for smooth functions the new derivative yields objects that are outside from $\DO'$. Finally, as we have seen, the resulting formulation has to be projected down to $\DO'$ in order to yield physically meaningful results.

(2) Kurasov's formulation of the $n$-th order derivative operator with s.a. interface conditions requires the use of infinite coupling constants in order to model all possible cases \cite{KurasovBoman}. In Problem 1 we have a more general situation, with a more general differential operator, and arbitrary linear interface conditions (s.a. and non s.a.). It is possible that Kurasov's approach can be used to determine an explicit formulation of the ODE$_2$ in some particular cases, but it seems less likely that a general solution can be obtained. In any case, ODEs with singular coefficients have never been studied in the literature using Kurasov's formalism, so it would be an interesting problem to determine which cases can be solve using this approach.

(3) If we consider the {\it direct} problem of formulating a general linear ODE (of the form (\ref{ODE1})) with singular coefficients and admit the possibility of singular solutions, then Kurasov's approach does not seem to be possible because, in general, even the term $a_0 \psi$ is ill-defined.  In a future paper we will address this problem for a class of singular coefficients using the intrinsic approach. We will see that in the general case the solutions of these ODEs can be discontinuous or singular.  

\bigskip

\noindent\textbf{Acknowledgements}. Cristina Jorge was supported
by the Ph.D grant SFRH/BD/85839/2012 of the Portuguese Science
Foundation (FCT). N.C. Dias and J.N. Prata were supported by the
Portuguese Science Foundation (FCT) under the grant
PTDC/MAT-CAL/4334/2014.

\vspace{0.5cm}

*******************************************************************

\textbf{Author's addresses:}

\begin{itemize}
\item \textbf{Nuno Costa Dias}\footnote{Corresponding author} and \textbf{Jo\~ao Nuno Prata:
}Grupo
de F\'{\i}sica Matem\'{a}tica, Universidade de Lisboa, Campo Grande, Edif\'{\i}cio C6, 1749-016 Lisboa, Portugal and Escola Superior N\'autica Infante D. Henrique, Av. Engenheiro Bonneville Franco, 2770-058 Pa\c{c}o de Arcos, Portugal.

\item \textbf{Cristina Jorge}: Departamento de Matem\'{a}tica.
Universidade Lus\'{o}fona de Humanidades e Tecnologias. Av. Campo
Grande, 376, 1749-024 Lisboa, Portugal and Grupo de F\'{\i}sica
Matem\'{a}tica, Universidade de Lisboa, Campo Grande, Edif\'icio C6,
1749-016 Lisboa, Portugal.

\end{itemize}

\vspace{0.3cm}

\small

{\it E-mail address} (NCD, Corresponding author): ncdias@meo.pt

{\it E-mail address} (CJ): cristina.goncalves.jorge@gmail.com

{\it E-mail address} (JNP): jnprata@FC.UL.PT \\

\normalsize

\vspace{0.3cm}

**************************************************************************


\begin{thebibliography}{100}

\bibitem{Albeverio1} Albeverio, S., Gesztesy, F., H\"ogh-Krohn, R., Holden, H. {\it Solvable Models in Quantum Mechanics}, 2nd ed., (AMS, Chelsea, 2005).

\bibitem{Albeverio2} Albeverio, S., Kurasov, P., {\it Singular perturbations of differential operators and solvable Schr\"odinger type operators},
(Cambridge University Press, 2000).

\bibitem{Bag95} Bagarello, F. Multiplication of distribution in one dimension: possible approaches and applications to $\delta$- function and its derivatives. Journal of Mathematical Analysis and Applications 196, 885–901, 1995.

\bibitem{Brasche} Brasche, J., Nizhnik, L. One-dimensional Schr\"odinger operators with general point interactions. Meth. Func. Anal. Topol. \textbf{19}, 4 - 15, 2013. 

\bibitem{Cad08} Caddemi, S., Cali\'o, I. Exact solution of the multi-cracked Euler-Bernoulli column. International Journal of Solids and Structures 45.5 : 1332-1351, 2008.


\bibitem{Col84}  Colombeau, J. F. New generalized functions and multiplication of distributions, North-
Holland Math. Studies, Vol. 84, North-Holland, Amsterdam, 1984.

\bibitem{Col85} Colombeau, J. F. Elementary introduction to new generalized functions, North-Holland
Math. Studies, Vol. 113, North-Holland, Amsterdam, 1985.


\bibitem{DP09} Dias, N.C., Prata, J.N. A multiplicative product of distributions and a class of ordinary differential equations with distributional coefficients, J. Math. Anal. Appl.\textbf{359}, 216-228, 2009.

\bibitem{DPP11} Dias, N.C., Posilicano, A., Prata, J.N.
Self-adjoint, globally defined Hamiltonian operators for systems
with boundaries, {\it Comm. Pure Appl. Anal.} {\bf 10}, no.6
(2011) 1687-1706.

\bibitem{DPJ16} Dias, N. C., Jorge, C., Prata, J. N. One-dimensional Schr\"odinger operators with singular potentials: A Schwartz distributional formulation. Journal of Differential Equations, \textbf{8}(260), 6548-6580, 2016.

\bibitem{DJP16-2} Dias, N.C., Jorge, C, Prata, J.N. An existence and uniqueness result about algebras of Schwartz distributions. To be submitted.

\bibitem{Exner} Exner, P., Neidhardt, H., Zagrebnov, V. Potential approximations to $\delta'$: an inverse Klauder phenomenon with norm resolvent
convergence, {\it Comm. Math. Phys.} {\bf 224} (2001) 593-612.

\bibitem{Golovaty} Golovaty, Y.D., Hryniv, R.O. Norm resolvent convergence of sigularly scaled Schr\"odinger operators
and $\delta'$-potentials, {\it Proc. R. Soc. Edinb. A.} {\bf 143},
no.4 (2013) 791--816.

\bibitem{GKOS01} Grosser, M., Kunzinger, M., Oberguggenberger, M., Steinbauer, R. Geometric Theory of
Generalized Functions, volume \textbf{537} of Mathematics and its Applications. Kluwer Academic
Publishers, Dordrecht, 2001.

\bibitem{HH08} Haller, S., H\"ormann, G. . Comparison of some solution concepts for linear first-order hyperbolic differential equations with non-smooth coefficients. Publications de l'Institut Mathematique,\textbf{84}(98), 123-157, 2008.
\bibitem{Hor83}  H\"ormander, L. The Analysis of Linear Partial Differential Operators I. Springer-Verlag, Berlin,
Heidelberg, 1983.

\bibitem{HO07} H\"ormann, G., Oparnica, Lj. Distributional solution concepts for the Euler-Bernoulli beam
equation with discontinuous coefficients. Applic. Anal., \textbf{86}(11):1347 -- 1363, 2007.

\bibitem{HO09} H\"ormann, G., Oparnica, Lj. Generalized solutions for the Euler-Bernoulli model with distributional
forces. J. Math. Anal. Appl., \textbf{357}(1):142 - 153, 2009.


\bibitem{Kan98} Kanwal, R.P. Generalized Functions: Theory and Technique, 2nd edition, Birkh\"auser, Boston, 1998.

\bibitem{Kurasov} Kurasov, P. Distribution theory for discontinuous test functions and differential operators with generalized coefficients. J. Math. Anal. Appl., \textbf{201}, 297 - 323, 1996.

\bibitem{KurasovBoman} Kurasov, P., Boman, J. Finite rank singular perturbations and distributions with discontinuous test functions. Proc. Amer. Math. Soc. \textbf{126}, 1673 - 1683, 1998.   

\bibitem{Obe92} Oberguggenberger, M. Multiplication of Distributions and Applications to Partial Differential
Equations, volume \textbf{259} of Pitman Research Notes in Mathematics. Longman, Harlow, U.K., 1992.

\bibitem{Teschl} Teschl, G. Ordinary Differential Equations and Dynamical Systems. Graduate Studies in Mathematics, vol. 140,
American Mathematical Society, 2012.

\bibitem{Sarrico} Sarrico, C. Collision of delta-waves in a turbulent model studied via a distribution product,
{\it Nonlinear Anal.} {\bf 73}, no.9 (2010) 2868 - 2875.

\bibitem{Sch54} Schwartz, L. Sur limpossibilit\'e de la multiplication des distributions. C. R. Acad. Sci. Paris
S\'er. I Math., \textbf{239} , 847 - 848, 1954.


\end{thebibliography}
\end{document}